\documentclass[12pt,a4paper]{article}
\usepackage[latin1]{inputenc}
\usepackage{amsmath,amsthm}
\usepackage{amsfonts}
\usepackage{amssymb}
\usepackage{graphicx,pstricks}
\input xy
\xyoption{all}
\theoremstyle{definition} %%% for statements in roman typeface

 \newtheorem{definition}{Definition}[section]
 \newtheorem{remark}[definition]{Remark}

\theoremstyle{plain}      %%% for statements in italic typeface

 \newtheorem{proposition}[definition]{Proposition}
 \newtheorem{theorem}[definition]{Theorem}

\title{Geometric interpretation of simplicial formulas for the Chern-Simons invariant}
\author{Julien March\'e
\footnote{Centre de Math\'ematiques Laurent Schwartz,
\'Ecole Polytechnique,
Route de Saclay, 91128 Palaiseau Cedex, France,
email:\,\tt{marche@math.polytechnique.fr}}
}
\date{}
\DeclareMathOperator{\hol}{Hol}
\DeclareMathOperator{\id}{Id}
\DeclareMathOperator{\tr}{Tr}
\DeclareMathOperator{\map}{Map}
\DeclareMathOperator{\tru}{Trunc}

\newcommand{\C}{\mathbb{C}}
\newcommand{\Z}{\mathbb{Z}}

\renewcommand{\hom}{\rm Hom}
\newcommand{\R}{\mathbb{R}}

\renewcommand{\lg}{\langle}
\newcommand{\ld}{\rangle}

\newcommand{\boL}{\mathcal{L}}
\newcommand{\bog}{\mathfrak{g}}

\newcommand{\bob}{\mathfrak{b}}
\newcommand{\boR}{\mathcal{R}}

\newcommand{\su}{\rm{SU}(2)}

\newcommand{\SL}{\rm{SL}(2,\mathbb{C})}
\newcommand{\PSL}{\rm{PSL}(2,\mathbb{C})}
\renewcommand{\d}{{\rm d}}
\renewcommand{\phi}{{\varphi}}
\renewcommand{\epsilon}{{\varepsilon}}

\newcommand{\mat}[1]{\begin{bmatrix} #1 \end{bmatrix}}
\begin{document}
\maketitle

\section{Introduction}
Chern-Simons theory was first introduced by S. S. Chern and J. Simons in \cite{cs} as secondary characteristic classes: given a Lie group $G$ and a flat $G$-bundle $P$ over a manifold $M$, all Chern-Weil characteristic classes of $P$ has to vanish as $P$ have a flat connection whereas the bundle might be non trivial. The Chern-Simons functional is a non trivial invariant of $G$-bundles with connections.

This theory had at least two unexpected developments:
in his seminal article \cite{witten}, E. Witten located Chern-Simons theory in the context of quantum field theory. He developed a topological theory giving rise to a physico-geometric interpretation of the Jones polynomial. On the other hand, in Riemannian geometry, the Chern-Simons functional provides a geometric invariant of compact manifolds with values in $\R/2\pi^2$ which in the case of 3-dimensional hyperbolic geometry makes with the volume the real and imaginary part of a same complex-valued invariant, see \cite{du,klassen,yoshida}. The work of J. Dupont and W. Neumann (\cite{du,neumann1,neumann0,neumann2}, with contributions of C. Zickert (\cite{dz}) gave a combinatorial formula for the Chern-Simons invariant of a cusped hyperbolic 3-manifold (and more generally pairs $(M,\rho)$ where $M$ is  a closed 3-manifold and $\rho:\pi_1(M)\to \PSL$ satisfies some assumptions), these formula were extended in a quantum setting by S. Baseilhac and R. Benedetti (\cite{bb}). All these approaches are based on group homology considerations and seminal computations of Dupont, which can be found in \cite{du}.

The purpose of this article is to prove exactly the same formulas as W. Neumann obtained but in a more direct and geometric way. The main idea is to fill all tetrahedra of a triangulated 3-manifold with a connection as explicit as possible and compute the contribution of each tetrahedron to the Chern-Simons functionnal. I think that this method simplifies the standard approach as the use of group homology or Bloch group is no longer necessary and we do not need to prove any kind of 5-term relation as it is satisfied for simple geometric reasons. On the other hand, the usual complications as branchings and flattenings are still necessary but they are given a geometric interpretation and their introduction looks more natural than in the standard construction.

The article is organized as follows: in Section 2 we explain generalities on the Chern-Simons invariant, adapting it to our context. In Section 3, we present the combinatorial structure which we will place at each tetrahedron of a triangulation of $M$. The heart of the article is Section 4 were we define the connection on a tetrahedron and almost compute its invariant. Then Section 5 contains standard material for extending the local computation to the global setting. In Section 6, we finish the computation of Section 4 by analyzing the 5-term relation and give the example of the figure-eight knot complement.

This work was completed two years ago when trying to generalize the known formulas for hyperbolic geometry to complex hyperbolic geometry: this work remains to be done. I would like to thank E. Falbel for being at the origin of this work, providing me with motivations and discussions, S. Baseilhac for encouraging me to write this article and W. Neumann for his kind interest.

\section{Generalities on the Chern-Simons functionnal}

\subsection{Generalities}
For a general discussion on classical Chern-Simons theory, we refer to \cite{freed}. Nevertheless, we recall here everything that will be needed for our purposes.
Let $G$ be a complex Lie group and $\bog$ be its Lie algebra. Let $M$ be an oriented 3-manifold with boundary $\Sigma$ and $P$ be a principal $G$-bundle over $M$ with right $G$-action and flat connection $\alpha$. Assuming this bundle is trivial,  one can identify it with $M\times G$ and the flat connection $\alpha$ may be viewed as an element of $\Omega^1(M,\bog)$ satisfying the flatness equation $$\d\alpha+\frac{1}{2}[\alpha\wedge\alpha]=0.$$
Denote by $\Omega^1_{\rm flat}(M,\bog)$ the space of all flat connections. Considering different trivializations corresponds to the following action of the gauge group: $\map(M,G)$ acts on $\alpha$ on the right by the formula $\alpha^g=g^{-1} \alpha g+g^{-1}\d g$.

Recall that given a path $\gamma:[0,1]\to M$, we define the holonomy of $\alpha$ along $\gamma$ as the solution at $t=1$ of the order one equation $(\frac{\d}{\d t}\hol_{\gamma}^t\alpha )(\hol_{\gamma}^t\alpha)^{-1}=-\alpha(\frac{\d\gamma}{\d t})$ such that $\hol_{\gamma}^0\alpha=1$. We denote it by $\hol_{\gamma}\alpha$ and it satisfies the two following properties: 

\begin{itemize}
\item[-] $\hol_{\gamma}\alpha^g=g(\gamma(1))^{-1}(\hol_{\gamma}\alpha) g(\gamma(0))$.
\item[-] $\hol_{\gamma\delta}\alpha=\hol_{\delta}\alpha\hol_{\gamma}{\alpha}$ where $\gamma$ and $\delta$ are two composable paths, that is $\delta(0)=\gamma(1)$.
\end{itemize}

It is well-known that the set of isomorphic classes of $G$-principal bundles with flat connection is in bijective correspondence via the holonomy map to conjugacy classes of morphisms from $\pi_1(M)$ to $G$. 

In this article, $G=\PSL$. Its universal cover is $\SL$. Because $\pi_1(\PSL)=\Z_2$, not all principal $G$-bundles on $M$ are trivial. The obstruction of trivializing such a principal bundle is a class in $H^2(M,\Z_2)$. One can recover this class as the obstruction of extending the monodromy representation $\rho:\pi_1(M)\to \PSL$ to $\SL$ (see \cite{gol,culler} for instance). Here we will consider trivializable $\PSL$-bundles $P$, hence we will suppose that the monodromy lifts to $\SL$ but an important point is that neither the trivialization of $P$ nor the lift of the monodromy morphism are part of the data.
 
\subsection{The definition of the Chern-Simons functionnal}
Let $\lg\cdot,\cdot\ld$ be an invariant symmetric bilinear form on $\bog$. For a flat connection $\alpha$, we set
$$CS(\alpha)=\frac{1}{12}\int_M \lg \alpha\wedge[\alpha\wedge\alpha]\ld.$$
 A direct computation shows that for $g$ in $\map(M,G)$ one has:
\begin{equation}\label{cs}
CS(\alpha^g)=CS(\alpha)+c(\alpha,g)\quad\text{ where }
\end{equation}

$$c(\alpha,g)=\frac{1}{2}\int_{\Sigma}\lg g^{-1}\alpha g\wedge g^{-1}\d g\rangle - \int_M \lg g^{-1}\d g\wedge[g^{-1}\d g\wedge g^{-1}\d g]\rangle.$$

The second term of this equation may be interpreted in the following way: let $\theta$ be the left-invariant Maurer-Cartan form on $G$ and $\chi$ be the Cartan 3-form on $G$, that is $\chi=\frac{1}{12}\lg\theta\wedge [\theta\wedge\theta]\ld$. Then we set $W(g)=\int_M \lg g^{-1}\d g\wedge[g^{-1}\d g\wedge g^{-1}\d g]\rangle=\int_M g^* \chi$.

This is called the Wess-Zumino-Witten functionnal of $g$. Modulo 1, it only depends on the restriction of $g$ to $\Sigma$ provided that $\chi$ is the image of some element of $H^3(G,\Z)$ inside $H^3(G,\C)$.

One can check that $c(\alpha,g)$ is a 1-cocycle which allows to construct a $\C$-bundle $\boL_{\Sigma}$ over $\hom(\Sigma,G)/G$ by taking the quotient of $\Omega^1_{\rm flat}(\Sigma,\bog)\times \C$ by the following gauge group action: $(\alpha,z)^g=(\alpha^g,e^{2i\pi c(\alpha,g)}z)$.
The Chern-Simons invariant of $\alpha$ may be interpreted as an element of $\boL_{\Sigma}$ lying above the gauge equivalence class of the restriction of $\alpha$ to the boundary.
\subsection{Boundary conditions}
Suppose that $G=\PSL$ and let $PB$ be the quotient of the Borel subgroup $B$ of upper triangular matrices by the center $\{\pm 1\}$. We will denote by $\bob$ the Lie algebra of $B$.
Let us consider triples (P,s,l) where:
\begin{itemize}
\item[-] $P$ is a trivializable flat $G$-bundle over $M$.
\item[-] $s$ is a flat section over the boundary of $M$ of the fibration $P\times_G\C P^1$ associated to $P$  for the action of $\PSL$ on $\C P^1$.
\item[-] Given a loop $\gamma:S^1\to \partial M$, we define $\lambda(\gamma)\in \C^*$ as the holonomy along $\gamma$ of the tautological line bundle lying over $s$. Then, there exists a continuous lift $l(\gamma)\in \C$ such that $\exp(l(\gamma))=\lambda(\gamma)$ and such that $l(\gamma.\delta)=l(\gamma)+l(\delta)$ for loops $\gamma$ and $\delta$ with the same base point.
\end{itemize}
A reformulation of the last assertion is that on the boundary, the $\C^*$-bundle defined by $s$ lifts to a $\C$-bundle via the log map. Here, the lift is part of the data.
There is an obvious equivalence relation between two such triples and we define $\boR(M)$ as the set of equivalent classes. We give now a gauge-theoretical description of $\boR(M)$.
\begin{proposition}\label{equivalence}
The set $\boR(M)$ of equivalence classes of triples $(P,s,l)$ is isomorphic to the quotient $\Omega^{1}_{\partial}(M,\bog)/\map_{\partial}(M,G)$ where
\begin{itemize}
\item[-] $\Omega^{1}_{\partial}(M,\bog)$ is the set of connections $\alpha\in \Omega^1(M,\bog)$ satisfying the flatness equation and such that their restrictions to the boundary lie in $\bob$.
\item[-] The group $\map_{\partial}(M,G)$ consists in maps from $M$ to $G$ whose restriction to the boundary takes values in $PB$ and is homotopic to $0$.
\end{itemize}
\end{proposition}
\begin{proof}
One can construct maps between these two spaces inverse to each other.
Given a connection $\alpha\in \Omega^{1}_{\partial}(M,\bog)$ one can give a flat structure on the trivial bundle $M\times G$. The holonomy of $\alpha$ takes its values in $PB$ and hence the vector $\C\oplus\{0\}\subset\C^2$ is preserved on the boundary and hence gives a flat section of $M\times\C P^1$. Consider the map $\mu:\bob\to\C$ which associates to an element of $\bob$ its upper left entry. Then for any path $\gamma:S^1\to \partial M$ one has $\lambda(\gamma)=\exp(-\int_{\gamma}\mu(\alpha))$. Hence we set $l(\gamma)=-\int_{\gamma}\mu(\alpha)$.
Replacing $\alpha$ with $\alpha^g$ for $g\in\map_{\partial}(M,G)$, one obtains an equivalent flat $G$-bundle with the same trivialization on the boundary. By assumption on $g$, $\mu(\alpha^g)-\mu(\alpha)$ is an exact form and hence, the map $l$ is unchanged.

Reciprocally, given a triple  $(P,s,l)$, one can identify the trivializable flat $G$-bundle $P$ with $M\times G$ by choosing a trivialization. The section $s$ gives a section of the trivial $\C P^1$-bundle over the boundary. As it has degree 0, one can suppose up to the action of the full gauge group that $s$ is constant equal to the class of $\C\oplus\{0\}$. The flat connection on $P$ gives a form $\alpha$ in $\Omega^1_{\partial}(M,\bog)$. The functions $l$ and $-\int\mu(\alpha)$ on loops on the boundary may differ, but there is a unique element $g$ in $\map(\partial M, PB)$ up to homotopy such that $l=-\int \mu(\alpha^g)$. It is always possible to extend $g$ to an element of $\map(M,G)$: the connection $\alpha^g$ is the desired representant of the triple $(P,s,l)$.
\end{proof}

Using the gauge-theoretical description above, we can represent a triple $(P,s,l)$ by an element $\alpha\in \Omega^1_{\partial}(M,\bog)$. We have the following proposition:
\begin{proposition}
The map $CS:\boR(M)\to\C/\Z$ is well-defined.
\end{proposition}
\begin{proof}
We have to check that the formula which defines $CS(\alpha)$ for $\alpha\in \Omega^1_{\partial}(M,\bog)$ is invariant under the action of $\map_{\partial}(M,G)$ modulo 1. It is a consequence of the following fact: for $\alpha\in \Omega^1_{\partial}(M,\bog)$ and $g\in \map_{\partial}(M,G)$, $c(\alpha,g)\in \Z$.
Recall the following formula: $c(\alpha,g)=\frac{1}{2}\int_{\partial M}\langle g^{-1}\alpha g\wedge g^{-1}\d g\rangle - W(g)$. 
The group $PB$ is homotopically equivalent to $S^1$. For any surface $\Sigma$ and map $f:\Sigma\to S^1$, there is a 3-manifold $N$ with boundary $\Sigma$ and a map $\tilde{f}:N\to S^1$ which extends $f$. We apply this to the map $g:\partial M\to PB$. One can find a 3-manifold $N$ such that $\partial N=\partial M$ and a map $\tilde{g}:N\to PB$  extending $g$. We can compute the WZW invariant of $g$ with $N$ and we deduce that $W(g)=\int_N \tilde{g}^* \chi \mod \Z$ but $\chi$ restricted to $\bob$ is zero as $\bob$ is 2-dimensional and $\chi$ is a 3-form. Hence we obtain $W(g)\in \Z$.

The first term of $c(\alpha,g)$ is proportional to $\int_{\Sigma}\mu(\alpha)\wedge\mu(g^{-1}\d g)$ because the bilinear form is proportional to the trace of the wedge product. The form $\mu(\alpha)$ is closed and $\mu(g^{-1}\d g)$ is exact by hypothesis. Hence, their product is exact and the integral vanishes by Stokes formula. This proves that $c(\alpha,g)$ belongs to $\Z$ and $CS$ induces a well-defined map from $\boR(M)$ to $\C/\Z$.
\end{proof}

One may give a third interpretation of $\boR(M)$ which corresponds to the holonomy description of the moduli space of flat connections. To give this correspondence, let us choose a base point $x_i$ on each connected component $\Sigma_i$ of the boundary of $M$. We will denote  by $\Pi_1(M)$ the fundamental groupo\"id of $M$ with these base points. Let us consider the set of groupo\"id maps $\rho:\Pi_1(M)\to G$ whose restrictions to the boundary components of $M$ take values in $PB$.

We call a logarithm of $\rho$ a collection of homomorphisms $l_i: \pi_1(\Sigma_i,x_i)\to \C$ such that for all $i$ and $\gamma\in\pi_1(\Sigma_i,x_i)$ one has $\rho(\gamma)=\pm\mat{z&*\\0&z^{-1}}$ where $z=\exp(l_i(\gamma))$.
Finally, two representations $\rho,\rho':\Pi_1(M)\to G$ are said to be equivalent if there is a collection of matrices $g_i\in PB$ such that for any path $\gamma$ between $x_i$ and $x_j$, one has $\rho'(\gamma)=g_j^{-1}\rho(\gamma)g_i$.
One can check easily that $\boR(M)$ is in bijection with the set of equivalence classes of pairs $(\rho,l)$ such that $\rho$ lifts to $\SL$. The bijection consists as usual in associating to a connection $\alpha$ the morphism $\rho(\gamma)=\hol_{\gamma}(\alpha)$. Reciprocally, given $\rho$ and $l$, one can define a pair $(P,s,l)$ by the standard procedure: one defines $P=\widetilde{M}\times G/\Pi_1(M)$ where $\widetilde{M}$ is the disjoint union of the universal coverings of $M$ based at the $x_i$ and $\Pi_1(M)$ acts on both factors by covering transformations and by $\rho$ respectively. One easily constructs the section $s$, and the map $L$ is unchanged.

\subsection{The derivative of the Chern-Simons functionnal}
Let $\alpha$ be a flat connection on a 3-manifold $M$ with boundary $\Sigma$. Recall that $\alpha$ satisfies $\d\alpha+\frac{1}{2}[\alpha,\alpha]=0$ and that we defined $CS(\alpha)=\frac{1}{12}\int_M\lg \alpha\wedge[\alpha\wedge\alpha]\ld$.

By differentiating the flatness equation, one sees that a 1-form $b$ is tangent to $\Omega^1_{\rm flat}(\Sigma,\bog)$ at $\alpha$ if and only if one has $\d b+ [b,\alpha]=0$.
The derivative of $CS$ at $\alpha$ in the direction $b$ is then
\begin{align*}
D_{\alpha}CS(b)&=\frac{1}{4}\int_M \lg b\wedge[\alpha\wedge\alpha]\ld=
-\frac{1}{2}\int_M\lg b\wedge \d\alpha\ld\\
&=\frac{1}{2}\int_M(\d\lg b\wedge \alpha\ld-\lg\d b\wedge \alpha\ld)=\frac{1}{2}\int_{\Sigma}\lg b\wedge\alpha\ld+\frac{1}{2}\int_M\lg[b\wedge\alpha]\wedge \alpha\ld\\
&=\frac{1}{2}\int_{\Sigma}\lg b\wedge\alpha\ld+\frac{1}{2}\int_M\lg b\wedge[\alpha\wedge\alpha]\ld
\end{align*}
We deduce the equation below which will be useful in the sequel.
\begin{equation}\label{derive}
D_{\alpha}CS(b)=\frac{1}{2}\int_{\Sigma}\lg \alpha\wedge b\ld
\end{equation}
This equation may be interpreted by saying that the Chern-Simons functional is a flat map $CS:\boR(M)\to \boL_{\Sigma}$ for some natural (non flat) connection on $\boL_{\Sigma}$, see \cite{freed}.

\section{Local combinatorial data}
In this article, we will work in the framework of 3-dimensional real hyperbolic geometry and hence suppose that $G$ is $\SL$, $\bog$ consists in trace free $2\times 2$ matrices and we set $\lg A,B\ld=-\frac{1}{4\pi^2}\tr(AB)$. This implies that $\chi=\frac{1}{12}\langle\theta\wedge[\theta\wedge\theta]\rangle$ is the positive generator of $H^3(\SL,\Z)$ in $H^3(\SL,\C)$ as we can compute that $\int_{\su}\chi=1$.

\subsection{Elementary Polyhedron}

Let $\Delta$ be a simplicial tetrahedron, that is a set of 4 elements. We will often use  $x,y,z,t$ as variables describing the elements of $\Delta$ and write them without comma to have more compact expressions.

\begin{definition}
The polyhedron $P(\Delta)$ associated to $\Delta$ is a polyhedral complex whose vertices are parametrized by orderings of the elements of $\Delta$.

There are three types of edges:
 $E_1=\{xyzt,yxzt\}$, $E_2=\{xyzt,xzyt\}$,  $E_3=\{xyzt,xytz\}$ which consist in transposing 2 consecutive vertices.

There are three types of 2-cells:
\begin{itemize}
\item[-]$\{xyzt,yxzt,xytz,yxtz\}$ which appears 6 times (type edge).
\item[-]$\{xyzt,xzyt,zxyt,zyxt,yzxt,yxzt\}$ which appears 4 times (type face).
\item[-]$\{xyzt,xytz,xtyz,xtzy,xzty,xzyt\}$ which appears 4 times (type vertex).
\end{itemize}
There is one 3-cell whose boundary is the union of all the faces.
\end{definition}
The polyhedral complex is best seen in Figure \ref{poly} as a tetrahedron whose edges and vertices are truncated. The ordering associated to a vertex has the form $xyzt$ where $x$ is the closest vertex, $xy$ is the closest edge and $xyz$ is the closest face. There are different ways of realizing this polyhedron in $\R^3$. Let us fix one of them once for all.

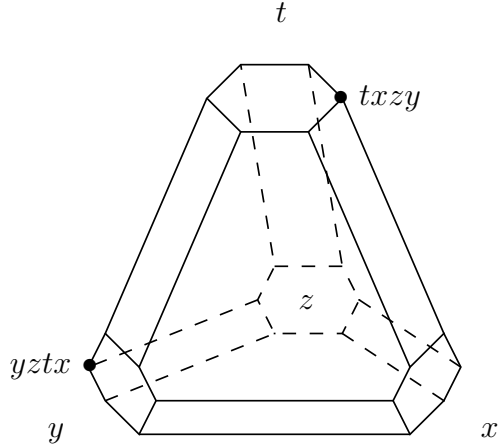
\begin{figure}
\begin{center}
\begin{pspicture}(-2,0)(3,3)
\includegraphics[width=5cm]{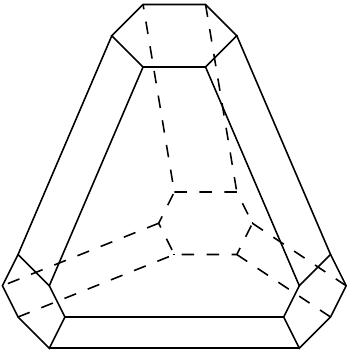}
\put(0.2,0){$x$}
\put(-5.5,0){$y$}
\put(-2.2,1.7){$z$}
\put(-2.5,5.5){$t$}
\put(-1.75,4.4){$\bullet$}
\put(-1.4,4.4){$txzy$}
\put(-5.05,0.85){$\bullet$}
\put(-6,0.85){$yztx$}
\end{pspicture}
\caption{The polyhedron $P(\Delta)$ \label{poly}}
\end{center}\end{figure}

\subsection{Configuration spaces and cocycles}\label{cocycle}

\begin{definition}
Let $\Delta$ be a simplicial tetrahedron. A {\it configuration} of $\Delta$ in $\C P^1$ is an injective map $\tau:\Delta\to \C P^1$.
\end{definition}

Composing by an element of $\PSL$, one gets an action of $\PSL$ on the set of configurations. Given a configuration $\tau:\Delta\to\C P^1$, we define a cross-ratio as a map $X_{\tau}: V(P(\Delta))\to \C\setminus\{0,1\}$ by the formula
$$X_{\tau}(xyzt)=\frac{(\tau_t-\tau_y)(\tau_z-\tau_x)}{(\tau_z-\tau_y)(\tau_t-\tau_x)}.$$

In this formula, we identified $\C P^1$ with $\C\cup\{\infty\}$. The map $X_{\tau}$ is a complicated (the 24 values of $X$ are determined by one of them) but invariant way for parametrizing configurations up to the action of $\PSL$.

Given a vertex $v=xyzt$ of $P(\Delta)$ and a configuration $\tau$, there is a unique configuration $\tau^v$ equivalent to $\tau$ via $\PSL$ action such that $\tau^v_x=\infty,\tau^v_y=0$ and $\tau^v_z=1$.
One can define a unique 1-cocycle $c\in Z^1(P(\Delta),\PSL)$ which satisfies for all edges $[v,v']$ of $P(\Delta)$ the following equality: $\tau^{v'}= c(v,v')\circ \tau^v$.

An easy computation gives the following formulas were we identified $\PSL$ with GL$(2,\C)/\C^*\id$.
$$c(xyzt,yxzt)=\mat{0&1\\ 1& 0},
\quad c(xyzt,xzyt)=\mat{1&-1\\ 0& -1}$$
$$c(xyzt,xytz)=\mat{1&0\\ 0& X_{\tau}(xyzt)}$$

\subsection{Orientation, branching and flattening}
\begin{definition}
An orientation of a simplicial tetrahedron $\Delta$ is a choice of ordering of its vertices up to even permutation.
\end{definition}
Notice that it is equivalent to the choice of an orientation of the realization of $\Delta$ or $P(\Delta)$.

\begin{definition}
A branching $b$ of a simplicial tetrahedron $\Delta$ is a choice of ordering of its vertices, or equivalently, a choice of vertex in $P(\Delta)$.
\end{definition}

Given a branching on $\Delta$, one can define an orientation of the edges of $P(\Delta)$ by orienting an edge from the vertex with lower ordering to the vertex with higher one with respect to the lexicographical ordering given by the branching. Of course a branching induces an orientation of $\Delta$, but both notions are introduced for independent purposes so we will ignore this coincidence.

The following definition and proposition are a variant of the constructions of \cite{neumann2}, Section 2. We give the details for completeness.
\begin{definition}
A flattening of a triple $(\Delta,b,\tau)$ where $\Delta$ is a simplicial tetrahedron, $b$ a branching on $\Delta$ and $\tau$ is a configuration of $\Delta$ in $\C P^1$ is a map $L:V(P(\Delta))\to \C$ such that the following relations are satisfied:
\begin{itemize}
\item[-] $L(xyzt)=-L(yxzt)=-L(xytz)$
\item[-] $L(xyzt)+L(xzty)+L(xtyz)=i\pi$ whenever one has $y<z<t$.
\item[-] $\exp(L(xyzt))=X_{\tau}(xyzt)$.
\end{itemize}
\end{definition}
The last equation shows that the data $L$ allows us to recover the map $X_{\tau}$. Hence, when dealing with flattenings, we will omit $\tau$ and $X$.

\begin{proposition}
The set $\boL(\Delta,b)$ of flattenings of a branched simplicial tetrahedron $(\Delta,b)$ is a Riemann surface. Given any vertex $xyzt$ of $P(\Delta)$ the map $L\mapsto \exp(L(xyzt))$ identifies $\boL(\Delta,b)$ with the universal abelian cover of $\C\setminus\{0,1\}$.
\end{proposition}
\begin{proof}
Suppose one has $\Delta=\{x,y,z,t\}$ with branching $x<y<z<t$. Then the first relation shows that $L$ reduces to the data $l_1=L(xyzt),l_2=L(xzty),l_3=L(xtyz)$ and $l'_1=L(ztxy),l'_2=L(tyxz),l'_3=L(yzxt)$.
By the third relation, for any $j\in\{1,2,3\}$ one has $\exp(l_j)=\exp(l'_j)$ so there is $k_j\in\Z$ such that $l'_j=l_j+2i\pi k_j$.
The second relation gives the following equations:
$l_1+l_2+l_3=i\pi, -l_1-l'_3-l'_2=i\pi, l_2+l'_3+l'_1=i\pi,-l_3-l'_2-l'_1=i\pi$.
Using the variables $k_j$ these equations reduce to $k_3=k_1=0$ and $k_2=-1$.
Hence only $l_1$ and $l_2$ are independant though not completely because of the following remaining identity: $\exp(l_2)=1/(1-\exp(l_1))$. This last expression shows that the map sending $L$ to $\exp(l_1)$ identifies $\boL(\Delta,b)$ with the universal abelian cover of $\C\setminus\{0,1\}$. The same is true for the other variables.
\end{proof}

\section{The connection of a polyhedron}

Let us fix once for all a function $\phi:[0,1]\to [0,1]$ which is smooth, satisfies $\phi(0)=0$ and $\phi(1)=1$ and whose derivative is non-negative and has compact support in $(0,1)$.

In the realization of a polyhedral complex, any oriented edge $e$ comes with a natural parametrization with $[0,1]$. By $\d\phi_e$ we will denote the 1-form on $e$ obtained by derivating $\phi$ through the natural parametrization of the edge. We will often drop the
subscript when the edge we are dealing with and its orientation are clear.

\subsection{The 1-skeleton}
Given $(\Delta,L,b)$ one can define a connection $\alpha$ on the 1-squeleton  of $P(\Delta)$ by the following formulas where the orientations of the edges are given by the branching.
The connection takes its values on the Lie algebra $\bog$ of $\SL$.

\begin{itemize}
\item[-] $\alpha(xyzt,yxzt)=\mat{ 0&i\pi/2\\ i\pi/2&0}\d\phi$,
\item[-] $\alpha(xyzt,xzyt)=\mat{ i\pi/2&-i\pi/2\\ 0&-i\pi/2}\d\phi$,
\item[-] $\alpha(xyzt,xytz)=\mat{L(xyzt)/2&0\\ 0&-L(xyzt)/2}\d\phi$.
\end{itemize}

\begin{proposition}\label{lift}
The holonomy of $\alpha$ gives a lift of $c$ to $\SL$.
\end{proposition}
\begin{proof}
The holonomies along the three corresponding type of edges are $M_1=\mat{ 0&-i \\ -i&0}$, $M_2=\mat{-i&i\\0&i}$ and
$M_3(L)=\mat{\exp(-L/2)& 0\\ 0 &\exp(L/2)}$ for $L=L(xyzt)$. These matrices are proportionnal to the corresponding values of the cocycle $c$.

Next, each face of $P(\Delta)$ gives an equation which should be satisfied.
The equation of type edge is equivalent to $M_1M_3(L(xyzt))=M_3(L(yxzt))M_1$. The equation of type face is equivalent to $M_1M_2 M_1=M_2M_1M_2$ and the equation of type vertex is equivalent to $$M_2M_3(L(xzyt))M_2=M_3(L(xtyz))M_2M_3(L(xyzt)).$$
 This last equation is a consequence of the relations satisfied by $X$ and $L$.
\end{proof}

\subsection{The 2-skeleton}

One need to explicit the restriction of the connection to the 2-cells of $P(\Delta)$.

{\bf Type edge:} let us identify the corresponding 2-cell with $[0,1]\times [0,1]$ such that the edges $[0,1]\times \{0,1\}$ have type $E_1$ and the edges $\{0,1\}\times [0,1]$ have type  $E_3$.
We denote by $s$ and $t$ the corresponding coordinates and suppose that the sides of the square are oriented in the direction of increasing $s$ and $t$. We set
\begin{equation}\label{edge}
\alpha=\mat{ 0&i\pi/2\\ i\pi/2&0}\d\phi_s+\frac{l}{2}\mat{\cos(\phi_s\pi)&i\sin(\phi_s\pi)\\-i\sin(\phi_s\pi)&-\cos(\phi_s\pi)}\d\phi_t.
\end{equation}
where $l$ is the value of $L$ at the vertex $(0,0)$.

{\bf Type face:} we fix once for all a smooth connection $\alpha$ on this 2-cell with the property that for any edge $e$ there is a standard neighborhood $U_e$ of the edge with projection $\pi_e$ on that edge such that $\alpha$ restricted to $U_e$ is equal to $\pi_e^*\alpha|_e$. This condition will ensure the smoothness of the total connection. The important point is that we will use the same connection for all type face 2-cells.

{\bf Type vertex:} this is the most difficult part as it seems that no preferred choice can be done without extra data. We remark at least that the connection on the boundary lies in the Borel subalgebra $\bob=\mat{*&*\\0&*}$. We extend the connection in an arbitrary way with the conditions that it takes its values in $\bob$ and that the same condition as for the type face cells holds in the neighborhood of the boundary.

It remains to fill the connection inside $P(\Delta)$. One can certainly do it as we have $\pi_2(\SL)=1$ so we fill it in an arbitrary way such that the connection is constant in some neighborhood of each face and we denote by $\alpha$ the flat connection thus obtained. Up to some controlled ambiguity, it depends only on $L$ and $b$.

\subsection{The CS-invariant and its derivative}\label{deriv}

Let $(\Delta,L,b,o)$ be a simplicial tetrahedron with flattening, branching and orientation. Let $\alpha$ be the flat connection on $P(\Delta)$ associated to $L$ and $b$.
One can compute its Chern-Simons invariant with respect to the orientation of $P(\Delta)$: we denote it by $CS(\Delta,L,b,o)=CS(\alpha)$.

\begin{proposition}
Let $(\Delta,L,b,o)$ be a 4-tuple as above. Then the function $CS: \boL(\Delta,b)\to \C/\Z$ is well-defined.
\end{proposition}
\begin{proof}
We need to show that $CS(\alpha)$ do not depend on the choices we made for constructing $\alpha$. The last choice we made was the filling of the interior of $P(\Delta)$.
As $P(\Delta)$ is simply connected, any other choice must be gauge equivalent with a gauge element which equals $\id$ at the boundary. The formula \eqref{cs} shows that the difference term is an integer.
Then, if we change the connection on faces of type vertex, the new connection is again gauge equivalent with a gauge element $g$ with values in $B$. For dimensional reasons, the restriction of $\chi$ to $B$ vanishes. Hence, the quantity $W(g)$ which expresses the difference also vanishes.
Finally, changing the connection on the faces of type face gives a difference with 4 contributions, one for each face. But these contributions have the same absolute values and signs depending on the relative orientations. Hence, the total difference also vanishes and the proposition is proved.\end{proof}

The aim of this section is to compute the derivative of $CS(\Delta,L,b,o)$ with respect to the geometric data $L$ by applying Formula \eqref{derive}.

Let us consider a 4-tuple $(\Delta,L,b,o)$. A tangent vector for the geometric parameter $L$ is a map $\delta L: V(P(\Delta))\to\C$ satisfying the two set of equations for all orderings $xyzt$ of $\Delta$.
\begin{itemize}
\item[-] $\delta L(xyzt)=\delta L(yxtz)=-\delta L(xytz)=-\delta L(yxzt)=\delta L(ztxy)$
\item[-] $\delta L(xyzt)+\delta L(xzty)+\delta L(xtyz)=0$ if $y<z<t$.
\end{itemize}
\begin{proposition}
One has $\delta CS(\Delta,L,b,o)=\frac{1}{8\pi^2}\left(l_2\delta l_1-l_1\delta l_2-i\pi \delta l_1\right)$ where $l_1=L(xyzt)$, $l_2=L(xzty)$ and $x<y<z<t$.
\end{proposition}
\begin{proof}
Given such a tangent vector, one can compute the corresponding tangent flat connection $\delta\alpha$. It vanishes identically on type face 2-cells as the connection on these cells do not depend on the geometric data.
Let $S$ be a type edge 2-cell, with parameter $s$ for the type E1 edge and parameter $t$ for the type E3 edge and with $l$ being the value of the $L$ map at the lower vertex $(0,0)$. The formula \eqref{edge} holds, and we deduce from it the formula
$$\delta\alpha=\frac{\delta l}{2}\mat{\cos(\phi_s\pi)&i\sin(\phi_s\pi)\\-i\sin(\phi_s\pi)&-\cos(\phi_s\pi)} \d \phi_t.$$
One computes that $\alpha\wedge\delta\alpha$ is trace free, hence $\lg\alpha\wedge\delta\alpha\ld$ vanishes and these 2-cell do not contribute either.

Finally, the remaining contributions come from the type vertex cells. Suppose that the branching and the orientation of $\Delta$ come from the ordering $x<y<z<t$ and let $S$ be the cell corresponding to the vertex $x$.

Let $a$ be the unique 1-form on $S$ such that $\alpha=\frac{1}{2}\mat{a&*\\ 0&-a}$: as $\alpha$ is flat, $a$ is closed. Let $\delta a$ be the 1-form corresponding to $\delta \alpha$, we choose a primitive of $\delta a$ that we denote by $A$ . Then, one computes

$$\frac{1}{2}\int_S\lg \alpha\wedge\delta\alpha\ld=\frac{1}{16\pi^2}\int_S\delta a\wedge a=\frac{1}{16\pi^2}\int_{\partial S}A\wedge a.$$

A computation shows that the final contribution of the cell is
$\frac{1}{32\pi^2}\delta l_3 (l_1+i\pi)-\frac{1}{32\pi^2}\delta l_1 (l_3+i\pi)$ where we have set $l_1=L(xyzt)$, $l_2=L(xzty)$ and $l_3=L(xtyz)$.
Summing up all vertices we obtain the formula of the proposition.
\end{proof}

To explain the relation with the standard dilogarithm, let us consider the case where
$\Delta=\{x,y,z,t\}$ with $x<y<z<t$, $\tau_x=\infty,\tau_y=0,\tau_z=u,\tau_t=1$ with $u\in (0,1)$.

Then one computes $X(xyzt)=1/u$, $X(xzty)=-u/(1-u)$ and $X(xtyz)=1-u$. Let us associate to such a configuration the flattening given by $l_1=-\log(u), l_2=\log(u)-\log(1-u)+i\pi, l_3=\log(1-u)$.

The derivative of $H(u)=CS(\Delta,L(u),b,o)$ with respect to $u$ is equal to $\frac{1}{8\pi^2}(\frac{\log(1-u)}{u}+\frac{\log(u)}{1-u})$. Recall that the Rogers dilogarithm is defined by $R(z)=\frac{1}{2}\log(z)\log(1-z)-\int_0^z\frac{\log(1-t)}{t} dt.$ We have $d H=\frac{-1}{4\pi^2}d R$, hence there exists $C$ such that $H(u)=C-\frac{R(u)}{4\pi^2}$. The determination of the constant will be a consequence of the 5-term equation of section \ref{comput}.

\section{Global data}
\subsection{Triangulations and their subdivisions}
\subsubsection{Abstract triangulations}
Let $(\Delta_i,o_i)_{i\in I}$ be a finite family of oriented simplicial tetrahedra. An orientation $o$ of a set $X$ is a numbering of the elements of $X$ up to even permutation. Any face of $\Delta$ gets an orientation by the convention that a positive numbering of the face followed by the remaining vertex is a positive orientation of $\Delta$.

We will call {\it abstract triangulation} a pair $T=((\Delta_i,o_i)_{i\in I},\Phi)$ where $\Phi$ is a matching of the faces of the $\Delta_i$'s reversing the orientation. If we realize this gluing with actual tetrahedra, then the resulting space $S(T)=\bigcup_{i}\Delta_i/\Phi$ may have cone singularities at vertices. To solve this problem, we truncate the triangulation. We do it in two stages: the first one will consists in truncating vertices, and the second one in truncating vertices and edges.

\subsubsection{Truncated triangulations}
For any simplicial tetrahedron $\Delta$, we define $\tru(\Delta)$ as the polyhedron whose vertices are pairs $(x,y)$ of distinct vertices of $\Delta$. As usual, we will use the more compact notation $xy$ for $(x,y)$. There are two type of edges, one has the form $(xy,yx)$ for distinct $x$ and $y$ and one has the form $(xy,xz)$ for distinct $x,y,z$. Then we add triangular cells of the form $(xy,xz,xt)$ and hexagonal cells of the form $(xy,yx,yz,zy,zx,xz)$. Finally we add one 3-cell: we may think of $\tru(\Delta)$ as the tetrahedron $\Delta$ truncated around vertices.

If $T=((\Delta_i,o_i)_{i\in I},\Phi)$, we define $\tru(T)=\bigcup_{i\in I} \tru(\Delta_i)$ where we identified the hexagonal faces following the indications of $\Phi$. All remaining faces are triangles which give a triangulation of the boundary of $\tru(T)$.
We will call {\it triangulation} of a compact oriented 3-manifold $M$ with boundary an abstract triangulation $T$ and an oriented homeomorphism $h:\tru(T)\to M$. This type of triangulation is often called a truncated ideal topological triangulation.

\subsubsection{Polyhedral triangulations}
Let $E(T)$ be the set of edges of $S(T)$: for any $e\in E(T)$, let $v(e)$ the set of pairs $(\Delta,\tilde{e})$ where $\tilde{e}$ is an edge of $\Delta$ projecting to $e$. We connect two points in $v(e)$ if the two corresponding simplicial tetrahedra are adjacent in $e$. Then $v(e)$ is a simplicial circle called the star of the edge $e$. We denote by $P(e)$ the polyhedra $D\times [0,1]$ such that $D$ is a regular polygon whose boundary is identified with $v(e)$.

We define the polyhedral subdivision of an abstract triangulation as the union $$P(T)=\left(\bigcup_{i\in I} P(\Delta_i)\cup\bigcup_{e\in E(T)}P(e)\right)/\Phi.$$ The 2-cells of type face of $P(\Delta_i)$ are identified together as prescribed by $\Phi$ and the 2-cells of type edge are glued to the 2-cells of the corresponding polyhedron $P(e)$. The 2-cells of type vertex are not glued. Hence, the realization of $P(T)$ is a manifold whose boundary has a cell decomposition with hexagons coming from type vertex 2-cells and polygons coming from $P(e)$.

\subsection{Branchings and flattenings}

Let $T=((\Delta_i,o_i)_{i\in I},\Phi)$ be an abstract triangulation. We will denote by $S(T)$, $\tru(T)$ and $P(T)$ and call singular, truncated and polyhedral triangulations the different topological realizations of the triangulation $T$. We remark that the set of vertices of $P(T)$ is in bijection with the set of ordered simplicial tetrahedra of $T$. We will denote loosely by $xyzt$ an ordered simplicial tetrahedron in $T$, and hence a vertex of $P(T)$.

{\bf Branching:}
A branching $b$ is an orientation of the edges of $S(T)$ such that the restriction to any simplex $\Delta$ in $T$ of the orientation of the edges is induced by an ordering of $\Delta$. This means than one can find a unique ordering on the vertices of $\Delta$ such that all the edges of $\Delta$ are oriented from the lower vertex to the higher one. This does not imply that there is a global ordering of the vertices of  $S(T)$.

{\bf Cross-ratio:}
A cross-ratio structure $X$ on $T$ is a map from the set of vertices of $P(T)$ to $\C\setminus\{0,1\}$ such that the following relations hold:
\begin{itemize}
\item[-] For all vertices of the form $xyzt$ one has
$$X(xyzt)=X(ztxy)=X(yxzt)^{-1}=X(xytz)^{-1}.$$
\item[-] In the same settings, one has $X(xzty)=1/(1-X(xyzt))$.
\item[-] For any oriented edge $e$ of $S(T)$ the product $\prod_i X(xyz_i z_{i+1})$ is equal to 1 where the simplices $xyz_iz_{i+1}$ involved in the product describe the set $v(e)$ where the oriented edge $xy$ projects to $e$ and the cyclic ordering of the points $z_i$ corresponds to the simplicial structure of $v(e)$.
\end{itemize}

{\bf Flattening:}
A flattening associated to a cross-ratio structure $X$ and branching $b$ on $T$ is a map $L$ from the set of vertices of $P(T)$ to $\C$ satisfying the following relations:
\begin{itemize} 
\item[-] For all vertices of the form $xyzt$ one has $L(xyzt)=-L(yxzt)=-L(xytz)$.
\item[-] Whenever one has $y<z<t$, the relation $L(xyzt)+L(xzyt)+L(xytz)=i\pi$ holds.
\item[-] In the same settings, one has $\exp(L(xyzt))=X(xyzt)$.
\item[-] For any edge $e$ of $S(T)$ the sum $\sum_i L(xyz_i z_{i+1})$ is equal to 0 where the simplices involved in the sum are the same that for the cross-ratios.
\end{itemize}

\subsection{Geometric meaning of cross-ratios and flattenings}

An abstract triangulation $T$ with cross-ratio structure $X$ induces a cocycle $c$ on the 1-skeleton of $P(T)$ with values in $\PSL$. This cocycle is given on $P(\Delta)$ by the formulas of the section \ref{cocycle} and no extra data needs to be defined as no edges have been added when passing from $\bigcup_i P(\Delta_i)$ to $P(T)$. The cocycle relations for all subpolyhedra $P(\Delta_i)$ are verified, and the relations coming from $P(e)$ for all edges $e$ are a consequence of the last equation of the definition of the cross-ratio structure.

Hence, if we set $M=P(T)$ for simplicity and choose base points $x_i$ on each connected component $\Sigma_i$ of the boundary of $M$, then a cross-ratio structure $X$ on $T$ produces a groupo\"id homomorphism $c:\Pi_1(M)\to G$ such that $c$ restricted to $\Sigma_i$ takes its values in $PB$.

This interpretation extends to the flattening $L$. Suppose there is a branching $b$ on $T$ and a flattening $L$ associated to $X$. Then we see as a consequence of the proposition \ref{lift} and of the last formula in the definition of flattening that the preceding cocycle $c$ and hence the corresponding groupo\"id homomorphism lifts to $\SL$.
Given a simplicial path $\gamma$ in $\Sigma_i$, one defines $L(\gamma)$ as the sum of the flattenings of the edges along the path. Then the pair $(c,L)$ is an element of $\boR(M)$. Hence, according to our geometric interpretation, cross-ratios and flattenings
on a triangulation $T$ are precisely the combinatorial data we need to define an element of $\boR(M)$ where $M$ is triangulated by $T$.

We would like to know which elements of $\boR(M)$ are obtained in this way: to this end we propose the following definition. Let $M$ be a 3-manifold with boundary and $(P,s,l)$ an element of $\boR(M)$.

\begin{definition}
Let $\gamma$ be an arc in $M$ whose ends lie in the boundary of $M$.
We will say that $\gamma$ is regular relatively to $(P,s,l)$ if the holonomy of the flat bundle $P\times_G\C P^1$ along $\gamma$ sends the section $s$ over the source point to an element distinct from the value of $s$ over the target point.
\end{definition}

Suppose that $M$ is triangulated, meaning that there exists an abstract triangulation $T$ and a homeomorphism $h:\tru(T)\to M$. Fix a branching on $T$ and let $(X,L)$ be a cross-ratio and flattening on $T$. As $P(T)$ and $\tru(T)$ are homeomorphic, the cocycle $c$ associated to $X$ gives a representation of $\Pi_1(M)$ into $\PSL$, and we set $\phi(X,L)=(c,L)\in \boR(M)$. We have the following proposition:

\begin{proposition}
For all elements $(P,s,l)\in\boR(M)$ such that all edges of $\tru(T)$ which do not lie on the boundary are regular, there is a pair $(X,L)$ such that $(P,s,l)=\phi(X,L)$.
\end{proposition}
\begin{remark}
This proposition is a variant of fairly well-known arguments, but we include it for completeness. The reader can refer to \cite{kashaev} for the notion of regularity and to \cite{dz} for the construction of flattenings.
\end{remark}
\begin{proof}
Let $(P,s,l)\in\boR(M)$ be such that all edges of $T$ are regular.

Let $\Delta=\{x,y,z,t\}$ be a simplex in $T$. The intersection of $\tru(\Delta)$ with the boundary of $M$ is a union of 4 triangles $T_x,T_y,T_z,T_t$. Pick any point $u$ in the boundary of $M$ and a path $\gamma$ from $u$ to the interior of $\tru(\Delta)$. Then, extend this path inside $\tru(\Delta)$ to 4 paths $\gamma_x,\gamma_y,\gamma_z,\gamma_t$ ending respectively in $T_x,T_y,T_z,T_t$. The parallel transport of the section $s$ at the end points gives a configuration of 4 points in the fiber of $P\times_G\C P^1$ over $u$. These  points are distinct by the assumption that all edges are regular. Call $X(xyzt)$ the cross-ratio of these 4 points. This number is well-defined and independent of $u$ and $\gamma$. One easily check that this construction defines on $T$ a cross-ratio structure. Only the third equation is not obvious but one can deduce it by choosing a fixed value of $u$ for all tetrahedra adjacent to the same edge.

Suppose that $(P,s,l)$ is represented by a connection $\alpha$ on the trivial bundle $M\times G$.
The 1-cocycle $c$ associated to $X$ is defined on $P(T)$ whereas the holonomy of $\alpha$ is a cocycle $\zeta$ on $\tru(T)$ with values in $\SL$. For each edge $e$ of $T$, we recall that we defined $P(e)=D\times [0,1]$ where $D$ was a regular polygon with boundary $v(e)$. Let us add to $P(e)$ the edge $\{0\}\times[0,1]$ and edges joining the origin to the vertices in $D\times\{0\}$ and $D\times\{1\}$. Let us call $P^+(T)$ this 1-skeleton. We can see $\tru(T)$ and $P(T)$ inside $P^+(T)$ as polyhedral decompositions of the same space. To see that $c$ and $\zeta$ are equivalent, it is sufficient to define a cocycle on $P^+(T)$ whose restriction on $P(T)$ and $\tru(T)$ is $c$ and $\zeta$ respectively. Let us choose arbitrarily the value of the cocycle of an edge in $D\times\{0\}$ joining the center to some vertex. Then, by cocycle relations and assumption on the restrictions of our cocycle, all other edges are determined, and these determinations actually define a cocycle on $P^+(T)$.

It remains to construct the flattening $L$ on $P(T)$. We do it in a geometric way, supposing that the triple $(P,s,l)$ is represented by a flat connection $\alpha$ on the bundle $M\times G$. The section $s$ is then trivial, as the tautological line bundle lying over $s$ which is equal to the bundle $\C\oplus \{0\}$.
For any oriented edge $e$ of $T$, let $\gamma$ be the corresponding path in $M$. Then, $(\hol_{\gamma}\alpha)(1,0)$ and $(1,0)$ are two independent vectors, hence their determinant is a non zero complex number. We choose a logarithm of this number that we call $l_e$. If we consider the edge with opposite orientation, then the determinant gets a minus sign, hence the logarithm differs by $i\pi\mod 2i\pi$. Thanks to the branching, one can specify an orientation of $e$. By convention, assume that the following relation holds: $l_{-e}=l_e-i\pi$ where $-e$ means the edge $e$ with negative orientation.

Moreover, recall from Proposition \ref{equivalence} that the datum $l$ in $(P,s,l)$ is given by the integral of $-\mu(\alpha)$ on the boundary. Using this formula, we may see $l$ as a 1-cocycle on the boundary of $P(T)$.
Given a simplicial path $\gamma$ in $\tru(T)$, we define $l_{\gamma}$ as the sum of the values of the function $l$ on the boundary edges of $\gamma$ and on the interior edges, keeping track of the orientation.
Given a vertex $xyzt$ of $P(T)$, we define $L(xyzt)$ in the following way: pick vertices $\tilde{x},\tilde{y},\tilde{z},\tilde{t}$ in the simplicial triangles $T_x,T_y,T_z,T_t$ and choose simplicial paths $\gamma_{uv}$ in $\tru(\{x,y,z,t\})$ for all $u,v$ in $\{\tilde{x},\tilde{y},\tilde{z},\tilde{t}\}$. Then we set
$L(xyzt)=l_{\gamma_{xz}}+l_{\gamma_{yt}}-l_{\gamma_{xt}}-l_{\gamma_{yz}}$. One can check that this defines a flattening recovering the logarithm $l$, see \cite{dz}.
\end{proof}

\section{Computation of the Chern-Simons invariant}\label{comput}

Suppose that one has a 3-manifold $M$ triangulated by $T$ and an element  $(P,s,l)$ of $\boR(M)$ represented by a triple $(b,X,L)$ where $b$ a branching of $T$, $X$ is a cross-ratio structure and $L$ is a flattening.

\begin{theorem}
$$CS(P,s,l)=\sum_{\Delta\in T} CS(\Delta,b,L,o)$$
\end{theorem}
This formula is an easy consequence of the fact that $M$ is a union of subpolyhedra. The triple $(P,s,l)$ is represented by an explicit flat connection on $M$ and the Chern-Simons invariant is an integral which decomposes as a sum of integrals over all subpolyhedra.
In the first part, we show that the polyhedra attached to the edges do not contribute to the sum and hence the formula reduces to a sum over elementary polyhedra. In the second part, we explain how the 5-term relation fits in this framework and give some applications and examples.

\subsection{Filling edges}
Around edges, one may need to glue back a polyhedron which has the form $P\times [0,1]$ where $P$ is a plane oriented polygon. By branching conditions, all edges of the form $\{v\}\times [0,1]$ where $v$ is a vertex of $P$ are oriented in the same direction and the restriction of the connection to it is $\mat{ 0&i\pi/2\\ i\pi/2&0}\d \phi_s$.

For any oriented edge $e$ of $P$, the two corresponding edges $e_0=e\times\{0\}$ and $e_1=e\times\{1\}$ are oriented in the same direction and the restriction of the connection to it has the form $\mat{ -L_{e}/2&0\\ 0&L_{e}/2}\d \phi_s$. Here $L_e$ is the value of $L$ at the starting point of $e$ which is opposite to its value at the target point.

By flattening conditions, the sum $\sum_e L_e$ vanishes where the edges $e$ are oriented in a compatible way with the boundary of $P$. One may fill the connection inside $P$ by taking any closed $\C$-valued 1-form $\omega$ which restricts to the corresponding form on each boundary segment. Precisely, we set $\alpha=\mat{ -\omega/2&0\\ 0&\omega/2}$ on $P\times\{0\}$ and $\alpha=\mat{ \omega/2&0\\ 0&-\omega/2}$ on $P\times\{1\}$.

One may fill the connection inside $P\times [0,1]$ by the condition that its restriction to segments of the form $\{v\}\times [0,1]$ is again $\mat{ 0&i\pi/2\\ i\pi/2&0}\d \phi_s$.
A direct computation shows that
$$\alpha=\omega \mat{\cos(\phi(s)\pi)&i\sin(\phi(s)\pi)\\-i\sin(\phi(s)\pi)&-\cos(\phi(s)\pi)}+\mat{ 0&i\pi/2\\ i\pi/2&0}\d \phi_s.$$
One computes directly from this expression that $CS(\alpha)=0$.

\subsection{The 5-term relation}\label{cinq}

As a consequence of such a gluing formula, one can deduce the well-known 5-term relation.
Let $X=\{x_0,\ldots,x_4\}$ be  a set with 5 elements. The union of the tetrahedra $\{x_0,x_1,x_2,x_4\}$ and $\{x_0,x_2,x_3,x_4\}$ is homeomorphic to the union of the following three ones: $\{x_1,x_2,x_3,x_4\}$, $\{x_0,x_1,x_3,x_4\}$ and $\{x_0,x_1,x_2,x_3\}$. Given a global order $x_0<\cdots<x_4$ and a map $L$ from the set of ordered 4-tuples of elements of $X$ to $\C$ satisfying the flattening relations, one deduce the following formula were $\Delta_i$ is the set $X\setminus\{x_i\}$

\begin{equation}\label{cinqtermes}
\sum_{i=0}^4 (-1)^{i}CS(\Delta_i,b,L)=0.
\end{equation}
The sign $(-1)^{i}$ takes into account the orientation of $\Delta_i$. One can deduce from it the precise formula for $CS$ in terms of the $L_i$, finishing the computation of Section \ref{deriv}. In terms of the function $H$ from $]0,1[$ to $\C$, one has for any $0<v<u<1$ the following equality:

$$H(u)-H(v)+H(v/u)-H(\frac{1-u^{-1}}{1-v^{-1}})+H(\frac{1-u}{1-v})=0.$$

Taking $u$ close to $v$, one finds that $\lim_{u\to 1}H(u)=0$. This gives us finally the expression $$H(u)=\frac{1}{24}+\frac{1}{8\pi^2}\int_0^{u}\left(\frac{\log(1-t)}{t}+\frac{\log(t)}{1-t}\right)\d t=\frac{1}{4\pi^2}(\frac{\pi^2}{6}-R(u)).$$

\subsection{An example}
Following Thurston, the figure eight knot complement is homeomorphic to the union of two tetrahedra without vertices.
We call $A$ and $B$ these two tetrahedra and we denote by$x,y,z,t$ their vertices (we use the same letters for the vertices of both tetrahedra).
We identify the faces of these tetrahedra in the only way which respects the colors of the arrows (black or white) and their directions.
Denote by $T$ the resulting abstract triangulation. A cross-ratio structure is determined by the two complex numbers $u=X_A(xyzt)$ and $v=X_B(xyzt)$ different from 0 and 1.
The complex $S(T)$ has two edges which gives the following relations:
\begin{align*}
X_A(tyxz)X_B(tyxz)X_A(xyzt)X_B(xzty)X_A(xzty)X_B(tzyx)=1\\
X_A(yzxt)X_B(yxtz)X_A(txzy)X_B(txzy)X_A(tzyx)X_B(yzxt)=1
\end{align*}
Translating into variables $u$ and $v$ both equations reduce to the equation
$$uv=(1-u)^2(1-v)^2$$
\begin{figure}
\begin{center}
\begin{pspicture}(-0.5,-0.5)(2,7)
\includegraphics{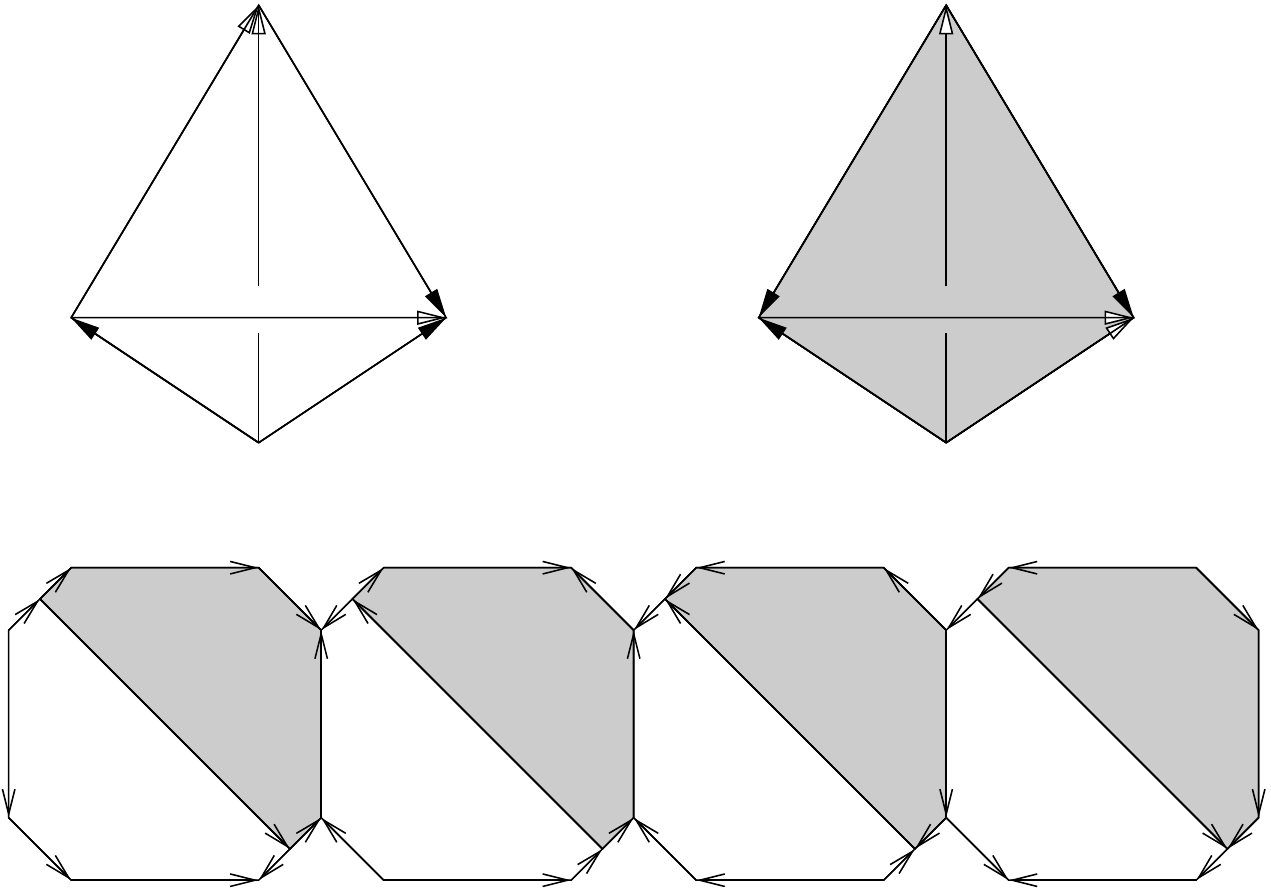}
\put(-12,8){$A$}
\put(-11,4.5){$x$}
\put(-8.2,6){$y$}
\put(-11,8.7){$z$}
\put(-12.5,6){$t$}
\put(-5,8){$B$}
\put(-4,4.5){$x$}
\put(-1.2,6){$y$}
\put(-4,8.7){$z$}
\put(-5.5,6){$t$}

\put(-11,2){$z$}
\put(-7.5,2){$x$}
\put(-4.5,2){$t$}
\put(-1.4,2){$y$}

\put(-12,1){$z$}
\put(-8.5,1){$x$}
\put(-5.3,1){$t$}
\put(-2.2,1){$y$}
\put(-12.5,0.5){$t$}
\put(-12.6,2.3){$x$}
\put(-10.7,0.3){$y$}

\put(-9.2,0.5){$t$}
\put(-9.3,2.3){$y$}
\put(-7.3,0.3){$z$}

\put(-6,0.5){$z$}
\put(-6.1,2.3){$y$}
\put(-4.2,0.3){$x$}

\put(-2.8,0.5){$z$}
\put(-2.9,2.3){$x$}
\put(-1,0.3){$t$}
\put(-12,2.9){$x$}
\put(-10.3,2.8){$y$}
\put(-10,0.8){$t$}

\put(-9,2.9){$t$}
\put(-7.2,2.8){$y$}
\put(-7,0.8){$z$}

\put(-5.8,2.9){$y$}
\put(-4.2,2.8){$x$}
\put(-3.8,0.8){$z$}

\put(-2.7,2.9){$z$}
\put(-0.8,2.8){$x$}
\put(-0.6,0.8){$t$}

\end{pspicture}
\caption{Decomposition of the figure eight knot complement and of its boundary}\label{huit}
\end{center}
\end{figure}
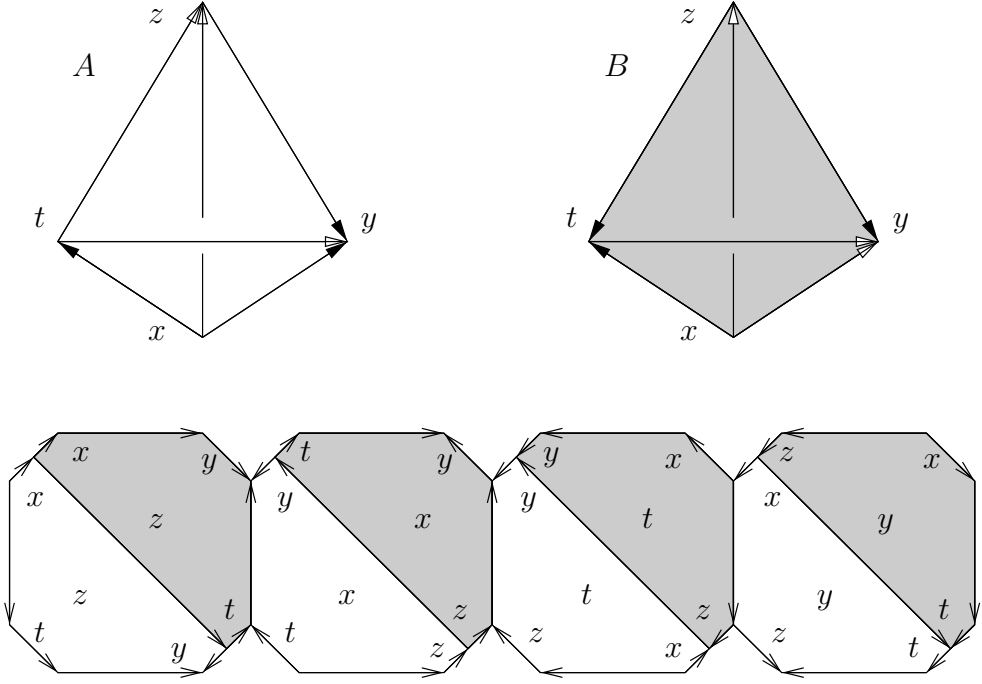
Let $b$ be the branching induced on $T$ by the arrows. One can check that the ordering induced on the vertices of $A$ is given by $x<t<z<y$ whereas the ordering induced on the vertices of $B$ is $x<z<t<y$.
Introduce the following variables:

\begin{center}\begin{tabular}{cc|cc}
$a_1=L_A(xtzy)$ & $a'_1=L_A(zyxt)$ & $b_1=L_B(xzty)$ & $b'_1=L_B(tyxz)$ \\
$a_2=L_A(xzyt)$ & $a'_2=L_A(ytxz)$ & $b_2=L_B(xtyz)$ & $b'_2=L_B(yzxt)$ \\
$a_3=L_A(xytz)$ & $a'_3=L_A(tzxy)$ & $b_3=L_B(xyzt)$ & $b'_3=L_B(ztxy)$ \\
\end{tabular}\end{center}

One has the following first set of relations:

\begin{tabular}{ccc}
$a'_1=a_1,$ & $b'_1=b_1,$ & $a_1+a_2+a_3=i\pi$  \\
$a'_2=a_2-2i\pi,$ & $b'_2=b_2-2i\pi,$ & $b_1+b_2+b_3=i\pi$ \\
$a'_3=a_3,$ & $b'_3=b_3$ &
\end{tabular}

and the edge equations:
$$-2a_2-a_3+2b_1+b_3+2i\pi=0,\quad -2a_1-a_3+2b_2+b_3-2i\pi=0$$

One can reduce these equations to the unknowns $a_1,a_2,b_1,b_2$ and the relation $a_1-a_2+b_1-b_2+2i\pi=0$.

Consider the curve $\alpha$ (resp. $\beta$) on the boundary of $P(T)$ represented by a vertical segment going upwards on the figure \ref{huit} (resp. an horizontal one from left to right). Then $\alpha$ and $\beta$ form a basis for the homology of the boundary. Let us express the logarithmic holonomy along $\alpha$ and $\beta$ in terms of our coordinates. One has $\mu(\alpha)=\frac{i\pi}{2}-a_1-b_1$ and $\mu(\beta)=2i\pi-2a_1-b_2$. The Chern-Simons invariant of the configuration that we are describing is finally equal to $CS(a_1,a_2,a_3)-CS(b_1,b_2,b_3)$.
We thus obtain the same formula as in \cite{neumann2}.

\end{document}